\documentclass[12pt, reqno]{article} 
\usepackage{amsfonts}
\usepackage{amsthm}
\usepackage{amscd}
\usepackage{amssymb}
\usepackage{graphics}
\usepackage{amsmath}
\usepackage[english]{babel}
\usepackage{hyperref}
\usepackage{bbm}
\usepackage{yfonts}
\usepackage{mathrsfs}
\usepackage{ upgreek }
\usepackage{color}
\usepackage{epsfig}
\usepackage{graphicx}

\usepackage{amsthm}
\usepackage{graphicx}
\usepackage{verbatim}
\usepackage{float}
\usepackage{fullpage}
\usepackage{amsfonts}
\usepackage{amscd}

\allowdisplaybreaks[4]

\DeclareFontFamily{OML}{rsfs}{\skewchar\font'177}
\DeclareFontShape{OML}{rsfs}{m}{n}{ <5> <6> rsfs5 <7> <8> <9>
rsfs7 <10> <10.95> <12> <14.4> <17.28> <20.74> <24.88> rsfs10 }{}
\DeclareMathAlphabet{\mathfs}{OML}{rsfs}{m}{n}
\newtheorem{prop}{Proposition}[section]
\newtheorem*{prop1}{Proposition}
\newtheorem{thm}[prop]{Theorem}
\newtheorem{lem}[prop]{Lemma}

\newtheorem{cor}[prop]{Corollary}
\newtheorem{conj}[prop]{Conjecture}

\newtheorem{defn}[prop]{Definition}
\newtheorem{rem}[prop]{Remark}
\newtheorem{prob}[prop]{Problem}

\newtheorem{ques}[prop]{Question}

\newcommand{\BE}{{\mathbb{E}}}

\newcommand{\BN}{{\mathbb{N}}}

\newcommand{\BQ}{{\mathbb{Q}}}
\newcommand{\BR}{{\mathbb{R}}}

\newcommand{\BZ}{{\mathbb{Z}}}

\newcommand{\tbP}{{\textbf{P}}}

\newcommand{\ind}{{\mathbbm{1}}}

\renewcommand{\prob}{{\bf P}}
\newcommand{\bae}{\begin{equation}\begin{aligned}}
\newcommand{\eae}{\end{aligned}\end{equation}}
\newcommand{\ev}{\mathbf{E}}
\newcommand{\pr}{\mathbb{P}}
\newcommand{\Z}{\mathbb{Z}}

\newcommand{\om}{{\omega}}

\newcommand{\ep}{{\epsilon}}
\newcommand{\lam}{{\lambda}}

\begin{document}

\numberwithin{equation}{section} \numberwithin{figure}{section}

\title{The need for speed :\\
Maximizing the speed of random walk in fixed environments}

\author{Eviatar B. Procaccia\footnote{Weizmann Institute of Science}, Ron Rosenthal \footnote{The Hebrew University of Jerusalem}}
\maketitle

\begin{abstract}
We study nearest neighbor random walks in fixed environments of
$\BZ$ composed of two point types : $(\frac{1}{2},\frac{1}{2})$ and
$(p,1-p)$ for $p>\frac{1}{2}$. We show that for every environment
with density of $p$ drifts bounded by $\lam$ we have
$\limsup_{n\rightarrow\infty}\frac{X_n}{n}\leq (2p-1)\lam$, where
$X_n$ is a random walk in the environment. In addition up to some integer
effect the environment which gives the greatest speed is given by
equally spaced drifts.
\end{abstract}

\maketitle



\section{Introduction}
The subject of random walks in non-homogeneous environments received much interest in recent decades. There has been tremendous progress in the study of such random walks in a random environment, however not much is known about random walks in a given fixed environment. In this paper we study the maximal speed a nearest neighbor random
walk can achieve while walking over $\BZ$, in a fixed environment
composed of two types of drifts, $\left(p,1-p\right)$ (i.e.
probability $p$ to go to the right, and probability $1-p$ to go to
the left) and $\left(\frac{1}{2},\frac{1}{2}\right)$.

A similar question in the continuous setting was posed by Itai
Benjamini and answered by Susan Lee. In \cite{lee1994optimal} Lee
proves that a diffusion process $dX_t=b(X_t)dt+dB_t$ on the interval
$[0,1]$, with $0$ as a reflecting boundary, $b(x)\geq 0$, and
$\int_0^1 b(x)dx=1$, has a unique $b$ which minimize the expected
time for $X_t$ to hit $1$, given by the step function $2\cdot
\ind_{[1/4,3/4]}$. This result is different in nature from the one
we get for the discrete case as the optimal environment in our case
is given by equally spaced drifts along $\BZ$. Notice that a major difference between Lee's setup and the one in this paper is that in the later the diffusion coefficient and drift are coupled. Another problem similar in spirit  is presented in \cite{berger2008solvency}, however the technical details are completely different. A related question for perturbation of simple random walk by a random environment of asymptotically small drifts, for which the recurrence/transience question becomes more involved is studied in \cite{menshikov2008logarithmic}.   

The question of this paper arose while the first author and Noam Berger tried to give a speed bound for a non Markovian random walks over $\BZ$ and the application will be published in \cite{MERW}.

In order to state the theorem we give a more precise definition of
the environments we study:

\begin{defn}
Given $\frac{1}{2}<p\leq 1$ and $0\leq \lam \leq 1$ we call
$\om:\BZ\rightarrow [0,1]$ a $(p,\lam)$ environment if the
following holds :
\begin{enumerate}
\item For every $x\in\BZ$ either $\om(x)=\frac{1}{2}$ or
$\om(x)=p$.
\item
\begin{equation} \label{erg}
\limsup_{n\rightarrow\infty}\frac{1}{n+1}\sum_{x=0}^{n}\ind_{\om(x)=p}=\lam
.\end{equation}
\end{enumerate}
\end{defn}

$~$\\

Throughout this paper we denote by $\{X_n\}_{n=0}^\infty$ a random walk on $\BZ$ (or
sub interval of it). In addition for a given environment
$\om:\BZ\rightarrow [0,1]$ and a point $x\in\BZ$ we denote by
$\tbP^x_\om$ the law of the random walk, which makes it into a
stationary Markov chain with the following transition probabilities
\[
\tbP^x_\om\left(X_{n+1}=y|X_n=z\right)=\begin{cases}
\omega(z) & \,\, y=z+1\\
1-\omega(z) & \,\, y=z-1\\
0 & \,\,\mbox{otherwise}\end{cases},
\]
and initial distribution
\[
\tbP^x_\om(X_0=x)=1.
\]\\

The goal of this paper is to study the maximal speed a random walk
in $(p,\lam)$ environments can achieve, i.e. the behavior of the
random variable $\limsup_{n\rightarrow \infty}\frac{X_n}{n}$.

We start with a simple observation regarding the random variable
$\limsup_{n\rightarrow \infty}\frac{X_n}{n}$:

\begin{lem}\label{lem_as_constant}
For every $(p,\lam)$ environment $\om$ and every $x\in\BZ$ the
random variable $\limsup_{n\rightarrow \infty}\frac{X_n}{n}$ is a
$\tbP^x_\om$ almost sure constant.
\end{lem}

The main theorem we prove is an upper bound on the speed of random
walks in $(p,\lam)$ environments:

\begin{thm}\label{Thm:Main_inequality}
For every $(p,\lam)$ environment $\om$ and every $x\in\BZ$
\[
\limsup_{n\rightarrow\infty}\frac{X_n}{n}\leq (2p-1)\lam ,\quad \tbP^x_\om ~~\text{a.s.}
\]
\end{thm}

$~$\\

As a result from the theorem we have the following corollary for random walks in random environments (RWRE):\\

\begin{cor}\label{Cor:RWRE_cor}
Let $P$ be a stationary and ergodic probability measure on
environments $\om$ of $\BZ$ such that  $P(\omega(0)=p)=\lam$ and
$P(\omega(0)=1/2)=1-\lam$ for some $0\le\lam\leq 1$ and
$1/2<p\leq 1$. Let $\{X_n\}$ be a RWRE with environment $\omega$
distributed according to $P$ (for a more precise definition of
RWRE see \cite{zeitouni2001lecture}), then
\[\lim_{n\rightarrow\infty}\frac{X_n}{n}\leq(2p-1)\lam.\]
for $P$ almost every environment $\om$ and $\prob^x_\om$ almost every
random walk in it.
\end{cor}

The main idea beyond the proof of Theorem \ref{Thm:Main_inequality}
is an exact calculation of some expected hitting times in a finite segment with a particular environment. We show that the expected hitting time of a random walk starting at the origin and reflected there, to the point $N$, where there are $k$ drift points between the origin and $N$ can be described by
\[
\ev^0_\om[T_N]=\frac{N^2}{(2p-1)\cdot k+1}+\left\langle
H_k(l-b),(l-b)\right\rangle ,\] where $l$ is the vector of drift
positions, $b$ is a fixed vector and $H_k$ is a $k$ by $k$ symmetric
positive definite matrix depending only on $p$. For the full
proposition and definitions see  Proposition
\ref{prop:finite_case_prop}. The last equation implies a lower bound
on $T_N$ and hence, eventually, an upper bound on the speed.

A natural question that arises is whether the inequality of Theorem \ref{Thm:Main_inequality} can be improved. In section \ref{sec:tight} we prove the following results:
\begin{prop}\label{prop:lameqm}
Let $m\in\BN$ and let $\lambda=\frac{1}{m}$. Let $\omega$ be an
environment s.t $\omega(i\cdot m)=p$, $\forall i\in\BZ$, and
$\omega(x)=\frac{1}{2}$, $\forall x\notin\{ i\cdot m:i\in\BZ\}$,
then a random walk $\{X_n\}$ in $\omega$ has the property
\[
\limsup_{n\rightarrow\infty}\frac{X_n}{n}=\lim_{n\rightarrow\infty}\frac{X_n}{n}=(2p-1)\lambda
.\]

\end{prop}

\begin{prop}\label{prop:lamthree}
For every $p$ and $\lam>0$, there exists a $(p,\lam)$ environment
$\omega$, and a constant $D(p)$ such that
\[
\limsup_{n\rightarrow\infty}\frac{X_n}{n}\ge(2p-1)\lam-D(p)\lam^3.
\]
\end{prop}
We also show the lower bound in proposition \ref{prop:lamthree} can
not be improved for all values of $p$.
\begin{prop}\label{prop_Lack_ot_tightness}
Let $\lam>0$ be of the form $\lam=\frac{n}{mn+l}$, such that $\lam\neq\frac{1}{k}$, for all $k\in\BN$. There exists a constant $D=D(n)>0$ such that
and every $(1,\lam)$ environment $\om$ we have
$\limsup_{N\rightarrow\infty}\frac{X_N}{N}\leq \lam-D\lam^3$.
\end{prop}

The structure of this paper is as follows: Section 2 deals with a
particular finite case of the problem which stands in the heart of
the proof of the infinite case. Section 3 contains the proof of
theorem \ref{Thm:Main_inequality}. Section 4 deals with RWRE. In
section 5 we discuss tightness of the result. In section 6 we prove
Lemma \ref{lem_as_constant}. Finally, in section 7 we give some
conjectures and open questions regarding the model.


\section{Finite environment with reflection at the origin}
\label{Finite_case}

We start by analyzing a finite variant of the problem. Consider nearest neighbor
random walks on subsets of $\BZ$ of the form
$\{0,1,\ldots,N\}$, with reflection at the origin, an absorbing
state at $N$, and the rest of the points are either
$(\frac{1}{2},\frac{1}{2})$ or $(p,1-p)$. More precisely we study the following environments :

\begin{defn}
Given  $N\in\BN$, $\frac{1}{2}<p\leq 1$ and $k\in\BN$ such that
$1\leq k \leq N-1$, we call $\om:\{0,1,\ldots,N-1\}\rightarrow [0,1]$
a $(N,p,k)$ environment on $\{0,1,\ldots ,N\}$ if there exists
$L=\{l_i\}_{i=1}^k\subset\BN$ such that
\[0<l_1<l_2<\ldots <l_k<N\]
and
\[
\om(x)=\begin{cases}
1 & \,\, x=0\\
p & \,\, x\in L\\
\frac{1}{2} & \,\, x\in\left\{ 1,\ldots,N-1\right\} \backslash
L\end{cases} .\]
\end{defn}

Throughout this section $T_N$ will denote the first time a random
walk in a $(N,p,k)$ environment $\om$ hits $N$, i.e,
$T_N=\min\{n\geq 0 ~:~ X_n=N\}$. In
addition we use the following notations :\\
$~$\\
1. $[N]=\{0,1,\ldots,N\}$\\
$~$\\
2. $l=(l_1,\ldots,l_k)$\\
$~$\\
and\\
$~$\\
3. $l_0=0$,~~ $l_{k+1}=N$.
$~$\\

The following is the main proposition of this section:\\

\begin{prop} \label{prop:finite_case_prop}
For every $(N,p,k)$ environment $\om$ we have
\begin{equation}
\ev^0_\om[T_N]\geq \frac{N^2}{(2p-1) k+1}.
\end{equation}
In addition there exists a $(N,p,k)$ environment which satisfies equality if and only if both $\frac{(2p-1)\cdot
N}{(2p-1)k+1}$ and $\frac{pN}{(2p-1)\cdot k +1}$ are integers.
Furthermore there exists a $k\times k$ positive definite symmetric matrix
$H_k$, with entries depending only on $p$, such that
\begin{equation}
\ev^0_\om[T_N]=\frac{N^2}{(2p-1) k+1}+\left\langle H_k(l-b),(l-b)\right\rangle,
\end{equation}
where $\left<,\right>$ denotes the standard inner product, and
$b=(b_1,\ldots,b_k)$ is the vector given by
\begin{equation}\label{eq:bis}
b_i=\frac{(2p-1) i+(1-p)}{(2p-1) k +1} N.
\end{equation}
\end{prop}

\begin{proof}
Define $v:[N]\rightarrow \BR$ by $v(x)=\ev^{x}_\om[T_N]$. By conditioning on the first step and using linearity of the
expectation one observes that $v$ satisfies the following
equations :
\bae\label{eq:lineq}
&v(0)=v(1)+1\\
&v(N)=0\\
&v(x)=\frac{1}{2}v(x+1)+\frac{1}{2}v(x-1)+1, \quad&\forall& x\in\left\{ 1,2,\ldots,N-1\right\} \backslash L\\
&v(x)=p\cdot v(x+1)+(1-p)\cdot v(x-1)+1, \quad&\forall& x\in L
.\eae

Restricting ourselves to an interval of the form $[l_{j-1},l_{j}],$
for some $1\leq j\leq N,$ we see that the solution to the equations
\[
v(x)=\frac{1}{2}v(x+1)+\frac{1}{2}v(x-1)+1,~~~~~~~~\forall
l_{j-1}<x<l_j,
\]
is given by $v(x)=-x^2+C_j\cdot x+D_j$ with $C_j$ and $D_j$ two
constants determined by the value of $v$ at $x=l_{j-1}$ and
$x=l_{j}$. Thus one can replace the equations in \eqref{eq:lineq} with the
following ones :
\bae
&v(0)=v(1)+1\\
&v(N)=0\\
&v(x)=-x^{2}+C_{j}x+D_{j}, \quad&\forall& x\in[l_{j-1},l_{j}]\,\,\,\,\,\,\,\,\forall 1\leq j\leq k+1\\
&v(l_{j})=p\cdot v(l_{j}+1)+(1-p)\cdot v(l_{j}-1)+1, \quad&\forall&1\leq j\leq k.
\eae

$~$\\
Solving those equations one  finds that
\bae
C_1&=0\\
C_j&=\frac{2(2p-1)}{1-p}\sum_{i=1}^{j-1}\left(\frac{1-p}{p}\right)^{j-i}l_i,\quad&\text{for }&2\leq j\leq k+1\\
D_{k+1}&=N^2-N\cdot\frac{2(2p-1)}{1-p}\sum_{i=1}^{k}\left(\frac{1-p}{p}\right)^{k+1-i}l_i\\
D_j&=N^2-\frac{2(2p-1)}{p}N\sum_{i=1}^{k}\left(\frac{1-p}{p}\right)^{k-i}l_i\\
&+\frac{2(2p-1)}{p}\sum_{i=j}^{k}l_{i}^2-\frac{2(2p-1)^2}{p^2}\sum_{i=j}^{k}\sum_{m=1}^{i-1}\left(\frac{1-p}{p}\right)^{i-m-1}l_il_m
,\quad&\text{for }& 1\leq j\leq k.
\eae
In particular we get that
\bae
f(l_1,\ldots,l_k):= \ev^{0}_\om[T_N]=
D_1&=N^2-\frac{2(2p-1)}{p}N\sum_{i=1}^{k}\left(\frac{1-p}{p}\right)^{k-i}l_i\\
&+\frac{2(2p-1)}{p}\sum_{i=1}^{k}l_{i}^2
-\frac{2(2p-1)^2}{p^2}\sum_{i=1}^{k}\sum_{m=1}^{i-1}\left(\frac{1-p}{p}\right)^{i-m-1}l_il_m.
\eae
Notice that the last function is a  polynomial of degree two in
$l_1,\ldots,l_k$.

One can check by substitution that the vector $b=(b_1,\ldots,b_k)$,
defined in \eqref{eq:bis}, is a solution to the equation $\textbf{grad}f(l)=0$, which makes $b$ into an extremum point of $f$.
In addition the Hessian of $f$ is constant (not depending on $N$ or $l_1,\ldots ,l_k$) and is given by the matrix
\[
H_k=-\frac{2(2p-1)^{2}}{p^{2}}\cdot\left(\begin{array}{ccccccc}
-\frac{2p}{2p-1} & 1 & \left(\frac{1-p}{p}\right) & \left(\frac{1-p}{p}\right)^{2} & \ldots & \ldots & \left(\frac{1-p}{p}\right)^{k-2}\\
1 & -\frac{2p}{2p-1} & 1 & \left(\frac{1-p}{p}\right) & \left(\frac{1-p}{p}\right)^{2} & \ldots & \left(\frac{1-p}{p}\right)^{k-3}\\
\left(\frac{1-p}{p}\right) & 1 & -\frac{2p}{2p-1} & 1 & \ldots & \ldots & \left(\frac{1-p}{p}\right)^{k-4}\\
\left(\frac{1-p}{p}\right)^{2} & \left(\frac{1-p}{p}\right) & 1 & \ddots\\
\vdots & \vdots &  &  & \ddots\\
\vdots & \vdots &  &  &  & \ddots & 1\\
\left(\frac{1-p}{p}\right)^{k-2} & \left(\frac{1-p}{p}\right)^{k-1}
&  &  &  & 1 & -\frac{2p}{2p-1}\end{array}\right)
,\]

$~$\\
We also define the matrix $M_k$ by
\[H_k\equiv
\left(-\frac{2(2p-1)^{2}}{p^{2}}\right)\cdot M_k.
\]
We notice that for $1\leq j \leq k,$ the $j^{\text{th}}$ principal
minors of $H_k$ and $M_k$ are exactly $H_j$ and $M_j$ respectively.

By subtracting the $(k-1)^{\text{th}}$ column and row multiplied by
$\frac{1-p}{p}$ of $M_k$ from the $k^{\text{th}}$ column and row
respectively, one gets the following recursion formula for the
determinant of $M_k$:
\[
\text{det}(M_k)= \left(-\frac{2p}{2p-1}
-\frac{2(1-p)}{p}-\frac{2(1-p)^2}{p(2p-1)}\right)\text{det}(M_{k-1})
-\left(1+\frac{2(1-p)}{2p-1}\right)^2\text{det}(M_{k-2}).
\]
Therefore, using induction one gets
\bae
&\text{det}(M_k)=(-1)^k\cdot\frac{k(2p-1)+1}{(2p-1)^k}\\
&\text{det}(H_k)=\frac{(k(2p-1)+1)2^k\cdot(2p-1)^k}{p^{2k}}.
\eae
Since $\text{det}(H_k)$ is positive for every $\frac{1}{2}<p\leq 1$ and
$k\in\BN$, it follows by Silvester's criterion (see \cite{Gilber:1991:PDM:115430.115439})
that $H_k$ is a positive definite matrix, and therefore $b$ is the
unique absolute minimum of $f=\ev^{0}_\om[T_N]$. Finally, by rearranging $f$ one can show that
\[
f(l)=\ev^0_\om[T_N]=\frac{N^2}{(2p-1)
k+1}+\left<H_k(l-b),(l-b)\right>.
\]

From the last formula we get $\ev^0_\om[T_N]\geq
\frac{N^2}{(2p-1) k+1}$ and equality holds if and only if $l=b$. One can see from the
definition of $b$ that such $l$ defines a $(N,p,k)$ environment if and only if both
$\frac{(2p-1) N}{(2p-1)k+1}$ and $\frac{pN}{(2p-1) k +1}$
are integers.
\end{proof}

Before turning to the infinite case we give a uniform bound on the norm of the matrices $H_k$, which will be used in Section \ref{sec:tight}.\\

\begin{lem}\label{prop:matrixnormbound}
There exists some finite positive constant $C=C(p)$ such that
\[
\sup_{k\in\BN}\|H_k\|_2\leq C.
\]
\end{lem}

\begin{proof}
Fix $k\in\BN$ and for $1\leq i\leq k$ denote by $r_k(i),c_k(i)$  the
$i^{th}$ row and column of the matrix $H_k$ respectively. We notice
that
\bae
\|r_k(i)\|_1 &=
\frac{2(2p-1)^2}{p^2}\cdot\left(\frac{2p}{2p-1}+\sum_{j=0}^{k-1-i}\left(\frac{1-p}{p}\right)^j+\sum_{j=0}^{i-2}\left(\frac{1-p}{p}\right)^j\right)\\
&\leq\frac{4(2p-1)}{p}+\frac{4(2p-1)^2}{p^2}\sum_{j=0}^{k-2}\left(\frac{1-p}{p}\right)^j\\
&\leq\frac{4(2p-1)}{p}+\frac{4(2p-1)^2}{p^2}\sum_{j=0}^{\infty}\left(\frac{1-p}{p}\right)^j=C(p)<\infty,
\eae
where we used the fact that $\frac{1}{2}<p\leq 1$ and therefore
$\frac{1-p}{p}<1$. The matrices $H_k$ are symmetric and therefore the same bound holds for $c_k(i)$.
We therefore get that :
\[
\|H_k\|_1\equiv\sup\{\|H_k v\|_1 : \|v\|_1=1\}=\max_{1\leq j\leq k}\sum_{i=1}^{k}|H_k(i,j)|\leq C(p),
\]
and
\[
\|H_k\|_\infty\equiv\sup\{\|H_k v\|_\infty : \|v\|_\infty=1\}=\max_{1\leq i\leq k}\sum_{j=1}^{k}|H_k(i,j)|\leq C(p).
\]
Using now the following estimate (which can be found for example in \cite{GVL96} Corollary 2.3.2)
\[
\|H_k\|_2\leq \sqrt{\|H_k\|_1\cdot \|H_k\|_\infty}
,\]
we get that for every $k\in\BN$
\[
\|H_k\|_2\leq C(p).
\]
\end{proof}


\section{Proof of the main theorem}

$~$\\
Fix $\frac{1}{2}< p \leq 1$ and $0\le \lam \leq 1$. We start with the following estimation of $\ev_\om^0[T_N]$ :

\begin{lem}\label{lem:cutenv}
Given two $\BZ$ environments $\om,\bar{\om}$ such that for every $x\in\BZ$,
${\omega}(x)\le\bar{\omega}(x)$. Denote by $T_n,\bar{T}_n$ the hitting times in the environments $\om,\bar{\om}$ respectively, then for every $n>0$, $T_n$ stochastically dominates $\bar{T}_n$, i.e $\prob_\om^0(T_n>t)\ge\prob^0_{\bar{\om}}(\bar{T}_n>t)$.
\end{lem}

\begin{proof}
This lemma follows from a standard coupling argument. Let $U_n\sim
U[0,1]$ be a sequence of i.i.d random variables. Let $\prob_{\om,\bar{\om}}$ be the
joint measure of two processes $X_n$ and $\bar{X}_n$ such that both the
processes at time $n$ move according to $U_n$ and the environments
$\omega$ and $\bar{\omega}$, i.e.
\begin{equation}
\prob_{\om,\bar{\om}}(X_{n+1}=x\pm
1,\bar{X}_{n+1}=\bar{x}\pm1|X_n=x,\bar{X}_n=\bar{x})=\prob_{\om,\bar{\om}}
\left(U_n\lessgtr\omega(x),U_n\lessgtr\bar{\omega}(\bar{x})\right),
\end{equation}
and
\begin{equation}
\prob_{\om,\bar{\om}}\left(\bar{X}_0=0 , X_0=0\right)=1.
\end{equation}
By this coupling whenever the processes meet at some point,
the random walk $\bar{X}_n$ has a higher
probability to turn right. We therefore obtain that $\prob_{\om,\bar{\om}}$ a.s
for every $n\in\BN$, $\bar{X}_n\ge X_n$, thus $\prob_{\om,\bar{\om}}$
a.s $\bar{T}_n\le T_n$.
\end{proof}

We turn now to prove the main theorem.

\begin{proof}[Proof of Theorem \ref{Thm:Main_inequality}]
Let $\ep>0$, and let $\om$ be a
$(p,\lam)$ environment. Since $\om$ is a $(p,\lam)$ environment there exists
$M\in\BN$ such that for every $N\geq M$ we
have
\begin{equation} \label{main_proof_good_density}
\frac{\#\{x\in[N] : \om(x)=p\}}{N}\leq \lam+\ep.
\end{equation}

$~$\\
For $N\geq M$ we define a new environment $\bar{\om}$ as follows :
\[
\bar{\om}(x)=\begin{cases}
\omega(x) & \,\, N\nmid x\\
1 & \,\, N\mid x\end{cases}
,\]
where $N\mid x$ is a shorthand for $N$ divides $x$.

Let $\bar{T}_n$ be the same hitting time distributed according to
the environment $\bar{\om}$. Since for every $x\in\BZ$ we have
$\om(x)\leq \bar{\om}(x)$ it follows, using Lemma
\ref{lem:cutenv}, that
\bae
\frac{T_{nN}}{nN}=\frac{1}{n}\sum_{k=1}^{n}\frac{{T}_{kN}-{T}_{(k-1)N}}{N}{\ge}
\frac{1}{n}\sum_{k=1}^{n}\frac{\bar{T}_{kN}-\bar{T}_{(k-1)N}}{N}\quad\prob_{\om,\bar{\om}}~\text{a.s.}
\eae

By the strong Markov property the random variables
$\{\bar{T}_{kN}-\bar{T}_{(k-1)N}\}_{k=1}^{\infty}$ are
independent (but for general environment not identically
distributed) and we wish to apply Kolmogorov's strong law of large numbers.\\
For $n\in\BN$ denote by $S_n$ the first hitting time of $n$ by a
symmetric simple random walk with reflection at the origin and starting at $0$. By Lemma \ref{lem:cutenv}, for every
$k\in\BN$ we have that $S_N$ stochastically dominates
$\bar{T}_{kN}-\bar{T}_{(k-1)N}$, and therefore
\[
\ev_{\bar{\om}}^0\left[\frac{\bar{T}_{kN}-\bar{T}_{(k-1)N}}{N}\right]\le\ev^0\left[\frac{S_N}{N}\right]=N<\infty,
\]
and
\[
\ev_{\bar{\om}}^0\left[\left(\frac{\bar{T}_{kN}-\bar{T}_{(k-1)N}}{N}\right)^2\right]
\le\ev^0\left[\left(\frac{S_N}{N}\right)^2\right]\le\frac{5}{3}N^2<\infty.
\]
The last relations are derived from the optional stopping theorem (see \cite{morters2010brownian} Theorem 12.20) and the fact that for a symmetric simple random walk $Y_n$,
\bae\nonumber
&Y_n^2-n,\\&
Y_n^4-6nY_n^2+3n^2+2n
\eae
are martingales. It therefore
follows by Kolmogorov's strong law of large numbers that
\begin{equation}\label{SLLN}
\lim_{n\rightarrow\infty}\frac{1}{n}\sum_{k=1}^n\frac{\bar{T}_{kN}-\bar{T}_{(k-1)N}}{N}
-\ev\left[\frac{1}{n}\sum_{k=1}^n\frac{\bar{T}_{kN}-\bar{T}_{(k-1)N}}{N}\right]=0,~\prob_{\bar{\om}}^0
\text{ a.s}.
\end{equation}

For $1\leq k\leq n$ we define $\lam_k$ to be the $p$'s density in
the interval $[(k-1)N,kN)$, i.e.
\[
\lambda_k=\frac{1}{N}\sum_{x=(k-1)N}^{kN-1}\ind_{\om(x)=p}.
\]
By \eqref{main_proof_good_density} we have
\begin{equation} \label{avarage_density}
\frac{1}{n}\sum_{k=1}^{n}\lambda_k=\frac{1}{n}\sum_{k=1}^{n}\frac{1}{N}\sum_{x=(k-1)N}^{kN-1}\ind_{\om(x)=p}
=\frac{\#\{x\in[nN-1] ~:~ \om(x)=p\}}{nN}< \lam+\ep.
\end{equation}
Notice that each of the segments $[(k-1)N,kN-1]$ of $\bar{\om}$ is a
$(N,p,\lam_kN)$ environment. It therefore follows by Proposition
\ref{prop:finite_case_prop} that \bae \frac{1}{n}\sum_{k=1}^n
\ev^0_{\bar{\om}}\left[\frac{\bar{T}_{kN}-\bar{T}_{(k-1)N}}{N}\right]
&\ge\frac{1}{n}\sum_{k
=1}^n\frac{1}{(2p-1)\cdot\lambda_k+\frac{1}{N}} \geq
\frac{n}{\sum_{k=1}^n(2p-1)\cdot\lambda_k+\frac{1}{N}}\\&\geq
\frac{1}{(2p-1)(\lambda+\epsilon)+\frac{1}{nN}}, \eae where the
second inequality follows from the inequality of arithmetic and harmonic means
and the third is by \eqref{avarage_density}. Thus,
\bae\liminf_{n\rightarrow\infty}\frac{T_n}{n}&\overset{\prob_{\om,\bar{\om}}
\text{ -a.s}}{\ge}
\liminf_{n\rightarrow\infty}\frac{1}{n}\sum_{k=1}^n\frac{\bar{T}_{kN}-\bar{T}_{(k-1)N}}{N}
=\liminf_{n\rightarrow\infty}\frac{1}{n}\sum_{k=1}^n\ev\left[\frac{\bar{T}_{kN}-\bar{T}_{(k-1)N}}{N}\right]\\
&\ge\lim_{n\rightarrow\infty}\frac{1}{(2p-1)(\lambda+\epsilon)+\frac{1}{nN}}=\frac{1}{(2p-1)(\lambda+\epsilon)}.
\eae
Since $\epsilon>0$ was arbitrary we obtain
\[
\liminf_{n\rightarrow\infty}\frac{T_n}{n}\ge\frac{1}{(2p-1)\lambda}.
\]
with the notation $\frac{1}{0}=\infty$. Now for $n\in\BN$ let $k_n$ be the unique random integers such that
\[T_{k_n}\le n<T_{k_n+1} .\]
Since $X_n<k_n+1$ we get that \bae
\frac{X_n}{n}-\frac{1}{n}\le\frac{k_n}{n}.
\eae
Thus
\bae
\limsup_{n\rightarrow\infty}\frac{X_n}{n}\le\limsup_{n\rightarrow\infty}\frac{k_n}{n}\le\limsup_{n\rightarrow\infty}\frac{k_n}{T_{k_n}}\le\limsup_{n\rightarrow\infty}\frac{n}{T_n}
=\frac{1}{\liminf_{n\rightarrow\infty}\frac{T_n}{n}}\le(2p-1)\lambda.
\eae
\end{proof}


\section{Application of the result to RWRE}

We turn now to prove corollary \ref{Cor:RWRE_cor}.

\begin{proof}[Proof of Corollary \ref{Cor:RWRE_cor}]
By ergodicity we obtain
\bae\label{eq:lamsum}
\lim_{n\rightarrow\infty}\frac{1}{n}\sum_{x=0}^{n}\ind_{\om(x)=p}=\lam, \quad P\text{ a.s.}
\eae
Define two random variables:
\bae
\bar{S}&=\sum_{i=1}^\infty\frac{1}{\omega{(-i)}}\prod_{j=0}^{i-1}\rho(-j)+\frac{1}{\omega(0)}\\
\bar{F}&=\sum_{i=1}^\infty\frac{1}{(1-\omega{(i)})}\prod_{j=0}^{i-1}\rho(j)^{-1}+\frac{1}{(1-\omega(0))},
\eae where $\rho(j)=\frac{1-\om(j)}{\om(j)}$. Since $\forall
i\in\BN$, $\omega(i)\ge1/2$, it follows that $\rho(i)^{-1}\ge1$ and
therefore $\bar{F}=\infty$, $P$ a.s. By Lemmas 2.1.9 and 2.1.12 of
\cite{zeitouni2001lecture}, if $E[\bar{S}]=\infty$ and
$E[\bar{F}]=\infty$ then
$\lim_{n\rightarrow\infty}\frac{X_n}{n}=0,\quad\prob_\om^0$  a.s for
$P$ almost every $\om$, and if $E[\bar{S}]<\infty$ then
$\lim_{n\rightarrow\infty}\frac{X_n}{n}=\frac{1}{\BE[T_1]}$ and
$\BE[T_1]<\infty$, where $\BE$ is the annealed expectation. By
\cite{zeitouni2001lecture} Lemmas 2.1.10 and 2.1.12 we have,
$\lim_{n\rightarrow\infty}\frac{T_n}{n}\overset{a.s}{=}\BE[T_1]$,
$\{T_{i+1}-T_i\}_{i=0}^{\infty}$ is a stationary and ergodic
sequence and
$\lim_{n\rightarrow\infty}\BE\left[\frac{T_n}{n}\right]=\lim_{n\rightarrow\infty}\frac{1}{n}\sum_{i=0}^{n-1}\BE\left[T_{i+1}-T_i\right]=\BE[T_1]$.
Let $\lam_n$ be the density of drifts in the interval $[0,n-1]$. By
Proposition \ref{prop:finite_case_prop} and \eqref{eq:lamsum} we
have
\[\BE[T_1]=\lim_{n\rightarrow\infty}\BE\left[\frac{T_n}{n}\right]\geq \lim_{n\rightarrow\infty} E\left[\frac{n}{(2p-1)\lam_nn-1}\right] =\frac{1}{(2p-1)\lam},\] and therefore
\[
\lim_{n\rightarrow\infty}\frac{X_n}{n}\le(2p-1)\lam ,\]
for $P$ almost every environment $\om$ and $\prob^x_\om$ almost every
random walk in it.
\end{proof}

It was pointed to the authors by Ofer Zeitouni that a trivial bound
to the speed exists to RWRE and one needs to check that this bound
is not better than the bound we get in Corollary \ref{Cor:RWRE_cor}. Note that the trivial bound is by no means tight.
By \cite{zeitouni2001lecture}
$\lim_{n\rightarrow\infty}\frac{X_n}{n}=\frac{1}{E\left[\bar{S}\right]}.$
Thus in order to get an upper bound on the speed a lower bound on
$E[\bar{S}]$ is needed. \bae
E[\bar{S}]&=\sum_{i=1}^\infty E\left[\frac{1}{\omega_{-i}}\prod_{j=0}^{i-1}\rho_{-j}\right]+E\left[\frac{1}{\omega_0}\right]\\
&=\sum_{i=1}^\infty E\left[e^{\sum_{j=0}^{i-1}\log\rho_{-j}-\log\omega_{-i}}\right]+E\left[\frac{1}{\omega_0}\right]\\
&\ge\sum_{i=1}^\infty e^{\sum_{j=0}^{i-1}E[\log\rho_{-j}]-E[\log\omega_{-i}]}+E\left[\frac{1}{\omega_0}\right]\\
&=p^{-\lambda}2^{1-\lambda}\sum_{i=1}^{\infty}e^{i\lambda\log\frac{1-p}{p}}+\frac{\lambda}{p}+(1-\lambda)2\\
&=p^{-\lambda}2^{1-\lambda}\frac{\left(\frac{1-p}{p}\right)^\lambda}{1-\left(\frac{1-p}{p}\right)^\lambda}+\frac{\lambda}{p}+(1-\lambda)2,
\eae where the inequality is Jensen's. Denote by
$S(p,\lambda)=p^{-\lambda}2^{1-\lambda}\frac{\left(\frac{1-p}{p}\right)^\lambda}{1-\left(\frac{1-p}{p}\right)^\lambda}+\frac{\lambda}{p}+(1-\lambda)2$.
In figure \ref{fig:NFSRWRE} we drew the difference between the bound
from Corollary \ref{Cor:RWRE_cor} and $S(p,\lambda)$. One can see
that for a large region of $p$ and $\lambda$ the bound archived in
Corollary \ref{Cor:RWRE_cor} is tighter.

\begin{figure}[H]
\begin{center}
\includegraphics[width=0.5\textwidth]{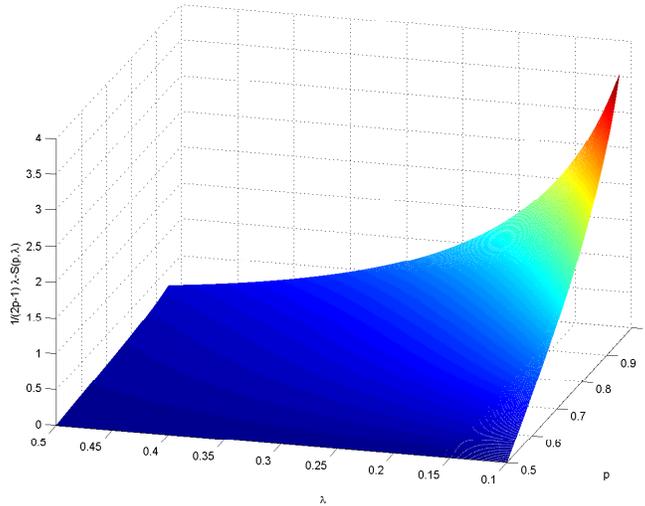}
\caption{Difference of bounds\label{fig:NFSRWRE}}
\end{center}
\end{figure}
\begin{rem}
Note that in the case of i.i.d RWRE an explicit value of the speed can be calculated and is equal to $\frac{(2p-1)\lam}{\lam+2p(1-\lam)}$ which is always smaller than $(2p-1)\lam$.
\end{rem}


\section{Tightness of the result}\label{sec:tight}
In this section we discuss  tightness of the result in the sense: Is there a $(p,\lambda)$ environment $\omega$, such that a
random walk $\{X_n\}$ in $\omega$ has the property
\[
\limsup_{n\rightarrow\infty}\frac{X_n}{n}=\lim_{n\rightarrow\infty}\frac{X_n}{n}=(2p-1)\lambda,\quad \prob_\omega\text{ a.s}
.\]
\subsection{Positive tightness}
\begin{prop1}$\textbf{{\textsl{\ref{prop:lameqm}.}}}$
Let $m\in\BN$ and assume $\lambda=\frac{1}{m}$. Let $\omega$ be an
environment defined by
\bae
\om(x)=\left\{\begin{array}{ll}
p\quad x\in m\BZ\\
\frac{1}{2}\quad\text{otherwise}
\end{array}\right.
,\eae
then a random walk $\{X_n\}$ in $\omega$ has the property
\[
\limsup_{n\rightarrow\infty}\frac{X_n}{n}=\lim_{n\rightarrow\infty}\frac{X_n}{n}=(2p-1)\lambda,\quad\prob_\omega\text{ a.s}
.\]
\end{prop1}
\begin{proof}
We prove this proposition by a direct calculation of the speed. \bae
\tbP_\om^0(T_m<T_{-m})&=p\left[\tbP_\om^1(T_m<T_0)+\tbP_\om^1(T_m>T_0)\tbP_\om^0(T_m<T_{-m})\right]\\
&+(1-p)\tbP_\om^{-1}(T_0<T_{-m})\tbP_\om^0(T_m<T_{-m}), \eae but
$\tbP_\om^1(T_m<T_0)=\tbP_\om^{-1}(T_0>T_{-m})=\frac{1}{m}$, thus
\[\tbP_\om^0(T_m<T_{-m})=p.\] Now $\ev_\om^0[T_m\wedge T_{-m}]=m^2$.
Consider $\omega$ as an environment with a constant drift $p$, such
that every jump takes on average $m^2$ steps. The speed of a random walk in an environment with 
constant drift $p$ at any point is $(2p-1)$. Thus the speed for $\omega$
is $(2p-1)\frac{m}{m^2}=(2p-1)\lambda$.
\end{proof}
We turn to prove a general tightness result.
\begin{prop1}\textsl{\bf \ref{prop:lamthree}.}
For every $\frac{1}{2}< p\le 1$ and $0<\lam\le1$, there exists a $(p,\lam)$ environment
$\omega$, and a constant $D(p)>0$ such that
\[
\lim_{n\rightarrow\infty}\frac{X_n}{n}\ge(2p-1)\lam-D(p)\lam^3,\quad\prob_\omega\text{ a.s}.
\]
\end{prop1}
\begin{proof}
First assume that $\lam\in\BQ$. We define the
environment $\omega$ by the positions $\{l_i\}$ of non-zero drifts on $\BZ$. For every $i\in\BZ$ let
$l_i=\left\lceil\frac{1}{\lam}\left(i+\frac{1-p}{2p-1}\right)\right\rceil$.
Note that since $\lam\le 1$ all the drift positions are distinct,
and  $\omega$ is indeed a $(p,\lam)$ environment. For every $N\in\BN$
we denote by $k=k(N)$, the number of drifts in the interval
$[0,N)$. Note that $\lim_{N\rightarrow\infty} \frac{k(N)}{N}=\lam.$ For a given $k\in\BN$ we denote
$b_{[k]}=(b_1,\ldots,b_k)$, where
\[
b_i=\frac{(2p-1)i+(1-p)}{(2p-1)k+1}N .\] By the Cauchy-Schwarz
inequality
\bae\langle H_k(l_{[k]}-b_{[k]}),(l_{[k]}-b_{[k]})\rangle&\le\| H_k(l_{[k]}-b_{[k]})\|_2\|(l_{[k]}-b_{[k]})\|_2\\
&\le\|H_k\|_2\|(l_{[k]}-b_{[k]})\|_2^2. \eae By Proposition
\ref{prop:matrixnormbound}, there exists a $C(p)$ such that
$\|H_k\|_2\le C(p)$ for every $k\in\BN$. Thus
\[
\lim_{N\rightarrow\infty}\frac{1}{N}\langle
H_k(l_{[k]}-b_{[k]}),(l_{[k]}-b_{[k]})\rangle\le
C(p)\lam\lim_{k\rightarrow\infty}\frac{1}{k}\|(l_{[k]}-b_{[k]})\|_2^2
.\] Notice that there exists a constant $C'$ (does not depend on $\lam$ or any other parameter) such that for $k$ large
enough $\|(l_{[k]}-b_{[k]})\|_\infty<C'$, thus
$\|(l_{[k]}-b_{[k]})\|_2\le C'\sqrt{k}$, therefore
\bae\label{eq:innerbound}
\lim_{N\rightarrow\infty}\frac{1}{N}\langle
H_k(l_{[k]}-b_{[k]}),(l_{[k]}-b_{[k]})\rangle\le C(p)C'\lam .\eae

Next we prove that for the environment $\omega$, the limit
$\lim_{n\rightarrow\infty}\frac{X_n}{n}$ exists. From Lemma 2.1.17
of \cite{zeitouni2001lecture} it is enough to show the limit
$\lim_{n\rightarrow\infty}\frac{T_n}{n}$ exists. Since $\lam$ is
rational there exists some $n_0\in\BN$ such that $\frac{n_0}{\lam}$ is an integer and therefore for every $i\in\BN$, $l_{n_0i}=\left\lceil i\frac{n_0}{\lam}+\frac{1-p}{\lam(2p-1)}\right\rceil= i\frac{n_0}{\lam}+\left\lceil\frac{1-p}{\lam(2p-1)}\right\rceil$. It follows that $\om$ is $n_0$ periodic
and therefore $\left\{T_{kn}-T_{(k-1)n}\right\}_{k=2}^\infty$ are
i.i.d. From the law of large numbers (Note that the first random variable in the sum is bounded and therefore negligible)
\[
\lim_{k\rightarrow\infty}\frac{1}{k}\sum_{j=1}^k\left[T_{j{n_0}}-T_{(j-1){n_0}}\right]=\ev[T_{n_0}],\quad\prob_\om\text{
a.s} .\] Now define $m_N$ to be the minimal integer such that
$m_N\cdot n_{0}\le N<(m_N+1)n_0$, then $\frac{T_{m_N\cdot
n_{0}}}{(m_N+1)n_{0}}\le\frac{T_N}{N}\le\frac{T_{(m_N+1)n_{0}}}{m_N\cdot
n_{0}}$, thus \bae\label{eq:limextn}
\lim_{N\rightarrow\infty}\frac{T_N}{N}=\lim_{k\rightarrow\infty}\frac{T_{kn_{0}}}{kn_{0}}=\frac{1}{n_{0}}\ev[T_{n_{0}}],\quad\prob_\om\text{
a.s} .\eae Note that if $\lim_{n\rightarrow\infty}\frac{X_n}{n}$
exists, it is the same for the environment $\omega$ and
$\bar{\omega}$ ($\omega$ with reflection at the origin), since the
random walk almost surely spends only a finite time  left of the
origin. By \eqref{eq:limextn} the limit,
$\lim_{N\rightarrow\infty}\frac{X_N}{N}$ exists and by Lemma 2.1.1
of \cite{zeitouni2001lecture},
$\lim_{N\rightarrow\infty}\frac{X_N}{N}=\lim_{N\rightarrow\infty}\frac{N}{\ev[T_N]}$.
We obtain from \eqref{eq:innerbound} and Proposition
\ref{prop:finite_case_prop} \bae
\label{density_speed_loewr_estimate}
\lim_{N\rightarrow\infty}\frac{X_N}{N}=\lim_{N\rightarrow\infty}\frac{N}{\ev[T_N]}\ge\frac{1}{\frac{1}{(2p-1)\lam}+C'C(p)\lam}\ge(2p-1)\lam
-C'C(p)(2p-1)^2\lam^3. \eae

Now For $\lam\notin\BQ$, let $\epsilon>0$ and let $0<\lam'<1$ be a rational number such
that $\lam-\ep<\lam'<\lam$. Define $\om$ to be the environment defined above for the rational density $\lam'$. Notice that $\om$ is a $(p,\lam')$ environment but
also a $(p,\lam)$ environment since $\lam'<\lam$. It follows from
\eqref{density_speed_loewr_estimate} that \bae
\lim_{N\rightarrow\infty}\frac{X_N}{N} &\geq (2p-1)\lam'
-C'C(p)(2p-1)^2\lam'^3 \\
&=(2p-1)\lam-C'C(p)(2p-1)^2\lam^3-(2p-1)(\lam-\lam')+C'C(p)(2p-1)^2(\lam^3-\lam'^3)\\
&\ge(2p-1)\lam-C'C(p)(2p-1)^2\lam^3-\left[(2p-1)+3C'C(p)(2p-1)^2\right]\ep,\eae taking $\epsilon$ small enough we obtain the result for some constant $D(p)>0$.

\end{proof}
\begin{rem}
Notice that for a rational $\lam$, by taking a uniform shift on the
environment $\om$ (shift right by an integer number uniformly chosen
between 0 and the period of $\om$), one gets an ergodic environment.
Thus from Proposition \ref{prop:lamthree} we get an example of a
RWRE which achieves the speed bound up to $\lam^3$.
\end{rem}
\subsection{Lack of tightness}
We now present an example where no environment achieves the speed
bound. This section also shows the bound in Proposition
\ref{prop:lamthree} can't be improved asymptotically.

Let $p=1$, $\lam=\frac{n}{mn+l}$ and assume $1<n\in\BN$, $l\in\BN$,
$0<l<n$ and $m\in\BN$. Note that the assumptions hold for every
rational number not of the form $\frac{1}{m}$ for some $m\in\BN,$ and
that $\left\lceil \frac{1}{\lam}\right\rceil=m+1$, $\left\lfloor
\frac{1}{\lam}\right\rfloor=m$.

We prove the following proposition :

\begin{prop1}\textsl{\bf \ref{prop_Lack_ot_tightness}.} Let $\lam>0$ be of the form $\lam=\frac{n}{mn+l}$ with the same
conditions as above. There exists a constant $D=D(n)>0$ such that
for every $(1,\lam)$ environment $\om$ we have
$\limsup_{N\rightarrow\infty}\frac{X_N}{N}\leq \lam-D\lam^3$.
\end{prop1}

We start by defining a family of environments $\Upsilon$ by the
following criteria : $\nu\in\Upsilon$ if the interval length between
two consecutive drifts in $\nu$ is either of length $m$ or of length
$m+1$ and there exists a limit to the density of drifts which equals $\lam$ . Under this assumption we can calculate the density of the two
different lengths. Denote by $\rho_i$ the density of intervals of
length $i$. Then under the assumptions we have
$\rho_m+\rho_{m+1}=\lam$ and $m \rho_m+(m+1)\rho_{m+1}=1$,
therefore \bae
&\rho_m =(m+1)\lam-1 \\
&\rho_{m+1}=1-m\lam.
\eae

\begin{prop} \label{prop:limit_proposition_Upsilon environment}
For every environment $\nu\in\Upsilon$ the limit $\lim_{n\rightarrow\infty}\frac{X_n}{n}$ exists, and
\[
\lim_{n\rightarrow\infty}\frac{X_n}{n}=\frac{1}{2m+1-m(m+1)\lam}
.\]
\end{prop}

\begin{proof}
Since $p=1$ we can write $T_N$ as
\[
\frac{T_N}{N}=\frac{1}{N}\sum_{i=1}^{r_m(N)}S_i(m)+\frac{1}{N}\sum_{i=1}^{r_{m+1}(N)}S_i(m+1)+\frac{1}{N}S,
\]
where $r_m(N),r_{m+1}(N)$ are the number of intervals of length $m$ and $m+1$ up to time $N$ respectively, $\{S_i(m)\}_i$ and $\{S_i(m+1)\}_i$ are two sequences of i.i.d random variables, where $S_i(j)$ is distributed as the first hitting time of $j$ by a simple random walk reflected at zero. $S$ is the first hitting time to the point $N-r_m(N) m-r_{m+1}(N) (m+1)$ of a simple random walk reflected at the origin, independently of both $\{S_i(m)\}_i$ and $\{S_i(m+1)\}_i$. Since $S$ is finite almost surely and since $\lim_{N\rightarrow\infty}\frac{r_j(N)}{N}=\rho_j$ for $j\in\{m,m+1\}$ we get by the strong law of large numbers that
\bae
\lim_{N\rightarrow\infty}\frac{T_N}{N}&=\rho_m\cdot \ev[S_1(m)]+\rho_{m+1}\ev[S_1(m+1)]\\
&=((m+1)\lam-1) m^2+(1-m\lam) (m+1)^2\\
&=2m+1-m(m+1)\lam.
\eae
Following the same argument as in Lemma 2.1.17 in \cite{zeitouni2001lecture} we get that $\lim_{N\rightarrow\infty}\frac{X_N}{N}$ exists and
\[
\lim_{N\rightarrow\infty}\frac{X_N}{N}=\frac{1}{\lim_{N\rightarrow\infty}\frac{T_N}{N}}=\frac{1}{2m+1-m(m+1)\lam}.
\]
\end{proof}

\begin{prop} \label{prop:environment_to_environment}
For every $(1,\lam)$ environment $\om$ there exists an environment
$\nu\in\Upsilon$ such that
\[
\liminf_{N\rightarrow\infty}\frac{T_N}{N}\geq \lim_{N\rightarrow\infty}\frac{T^\nu_N}{N},
\]
where $T^\nu_N$ are the hitting times in the environment $\nu$ and $T_N$ are the ones in the environment $\om$.
\end{prop}

\begin{proof}
Without loss of generality we assume that in the environment $\om$ the
limit \[\lim_{n\rightarrow\infty}\frac{1}{n}\sum_{x=1}^{n}\ind_{\om(x)=p}\]
exists and equals $\lam$. Indeed adding drifts to an environment only decreases the hitting times and we can always add drifts such that the limit of the density of the new environment exists. Let $\{\ep_j\}_{j\in\BN}$ a sequence of
positive numbers such that $\lim_{j\rightarrow\infty}\ep_j=0$.
Notice that for large enough $j\in\BN$ we have
$m\equiv\left\lfloor\frac{1}{\lam\pm\ep_j}\right\rfloor=\left\lfloor\frac{1}{\lam}\right\rfloor$.
and also
$m+1\equiv\left\lceil\frac{1}{\lam\pm\ep_j}\right\rceil=\left\lceil\frac{1}{\lam}\right\rceil$,
and so we assume this is true for every $j\in\BN$. We turn now to
define a sequence of environments $\zeta^n$ which will gradually
turn into an $\Upsilon$ environment. We assume without loss of
generality that $\om(0)=1$.

Fix some $N_1\in\BN$ large enough such that
for every $n\geq N_1-1 $ the density in the interval $[0,n]$ is between $\lam-\ep_1$ and $\lam+\ep_1$. In particular we have
\[
\lam-\ep_1<\lam_1\equiv\frac{1}{N_1-1}\sum_{i=1}^{N_1-1}\ind_{\om(x)=p}<\lam+\ep_1
\]
and also $\om(N_1)=p$. Note that by the assumptions
$\left\lfloor\frac{1}{\lam_1}\right\rfloor=m$ and
$\left\lceil\frac{1}{\lam_1}\right\rceil=m+1$.

Let us first analyze the environment in the interval $[0,N_1]$. We
denote by $k_1$ the number of drifts in the interval $(0,N_1)$, and
for $i\geq 1$ we denote by $r_i$ the number of intervals between two
consecutive drifts of length $i$ (We count in the interval length
the left drift but not the right one). We have the following
relations: \bae\label{eq:rieq}
N_1&=\sum_{i=0}^{N_1}ir_i\\
k_1+1&=\sum_{i=0}^{N_1}r_i\\
\ev[T_{N_1}]&=\sum_{i=0}^{N_1}i^2r_i.\\
\eae

Assume that there exist two indices $i<k$ such that $r_i,r_k>0$,
$k-i\ge2$ and either $k>m+1$ or $i<m$. By changing the location of
the drifts one can replace one interval of length $k$ and one of
length $i$ with intervals of length $k-1$ and $i+1$. By doing so one
gets a new environment with the same total length, same number
of drifts and with $\ev[T_N]$ smaller by $2(k-i-1)$. Since the
interval $[0,N_1]$ is finite, one can apply the last procedure only
finite number of times, and achieve a new environment $\zeta^1$.
Note that the environment $\zeta^1$ satisfy the following
conditions:
\begin{itemize}
\item For every $x\ge N_1$ we have $\om(x)=\zeta^1(x)$.
\item In the environment $\zeta^1$ inside the interval $[0,N_1]$
there are only intervals of length $m$ and $m+1$.
\end{itemize}
Indeed, the first claim is immediate from the fact the only changes
we made where in the interval $(0,N_1)$. For the second claim, note
that if in the end of the finite procedure one is left with an interval of length larger than $m+1$, then all the intervals are of length larger or equal to $m+1$ therefore the density is smaller than $\lam_1$. Same argument shows no intervals of length smaller than $m$ are left at the end of the procedure in the interval $(0,N_1)$.

Since each step of the procedure defining $\zeta^1$ decreased the value
of $\ev[T_{N_1}] $, and since $\om$ and $\zeta^1$ coincide for $x\ge N_1$ we get that $\ev[T_n^{(1)}]\leq \ev[T_n]$ for every $n\geq
N_1$, where $T_n^{(1)}$ is the first hitting time of  where $n$ in
the environment $\zeta^1$.

Let $N_2\in\BN$ be large enough so that $N_2>N_1$,
\[
\lam-\ep_2<\frac{1}{N_2-1}\sum_{i=1}^{N_2-1}\ind_{\om(x)=p}<\lam+\ep_2,
\]
and
\[
\lam-\ep_2<\lam_2\equiv\frac{1}{N_2-N_1-1}\sum_{i=N_1+1}^{N_2-1}\ind_{\om(x)=p}<\lam+\ep_2.
\]
Repeating the last procedure on the interval $[N_1,N_2]$ one can
define a new environment $\zeta^2$ such that:
\begin{itemize}
\item $\zeta^2(x)=\om(x)$ for every $x\geq N_2$.
\item In the interval $[0,N_1]$ the environments $\zeta^1$ and
$\zeta^2$ agree.
\item In the interval $[0,N_2]$ the length between
two consecutive drifts is either $m$ or $m+1$.
\item For every $n\geq N_2$ we have $\ev[T_n^{(2)}]\leq\ev[ T_n^{(1)}]\leq
\ev[T_n]$.
\item For every $n\geq N_1$ the density of the drifts in the
interval $(1,n)$ is between $\lam-2\ep_1$ and $\lam+2\ep_1$.
\end{itemize}

For the last point, notice that changing the order of intervals in $(N_1,N_2)$ does not change $\ev[T_{n}]$ for $n\ge N_2$. By rearranging the order of intervals we can ensure the last point is satisfied.

Repeating the last procedure and defining $\zeta^{j+1}$ from
$\zeta^{j}$ in the same way, we get a sequence of environments.
Finally define the environment $\nu$ by
\[
\nu(x)=\lim_{j\rightarrow\infty}\zeta^j(x), \quad \forall x\geq 0.
\]
This is well defined since for every $x\geq 0$ there exists
$j_0\in\BN$ such that for every $j\geq j_0$ the value of
$\zeta^j(x)$ is constant. From the definition of $\nu$ the
environment is indeed in the family $\Upsilon$.

Denote by $l_i$ the location of the $i^{th}$ drift to the right of
zero in the environment $\nu$ and $l_0=0$. In addition for every
$n\in\BN$ we define $k(n)$ to be the unique integer such that
$l_{k(n)}<n\leq l_{k(n)+1}$. It therefore follows that for every
$n\in\BN$ we have
\[
\frac{T_n}{n}=\frac{T_n-T_{l_{k(n)}}}{n}+\frac{1}{n}\sum_{i=1}^{k(n)}T_{l_{k(i)}}-T_{l_{k(i)-1}}
.\]

Since in the environment $\nu$ we only have intervals of length $m$
and $m+1$ we have
\[
\frac{T_n-T_{l_{k(n)}}}{n}\leq \frac{T_n-T_{n-m-1}}{n}
\]
and therefore
$\lim_{n\rightarrow\infty}\frac{T_n-T_{l_{k(n)}}}{n}=0$, $\prob_\om$ a.s.
Consequently we get that
\[
\liminf_{n\rightarrow\infty}\frac{T_n}{n}=\liminf_{n\rightarrow\infty}\frac{1}{n}\sum_{i=1}^{k(n)}T_{l_{k(i)}}-T_{l_{k(i)-1}},\quad\prob_\om\text{ a.s.}
\]
Since $k(n)$ as defined above equals to the number of drifts in the
interval $(0,n)$, we get from the construction of the environment
$\nu$ that $\lim_{n\rightarrow\infty}\frac{k(n)}{n}=\lam$. Thus we
get that
\[
\liminf_{n\rightarrow\infty}\frac{T_n}{n}=\liminf_{n\rightarrow\infty}\frac{\lam}{k(n)}\sum_{i=1}^{k(n)}T_{l_{k(i)}}-T_{l_{k(i)-1}},\quad\prob_\om\text{ a.s.}
\]
which by Kolmogorov strong law of large numbers equals to
\[
\liminf_{n\rightarrow\infty}\frac{\lam}{k(n)}\sum_{i=1}^{k(n)}\ev\left[T_{l_{k(i)}}-T_{l_{k(i)-1}}\right]
=\liminf_{n\rightarrow\infty}\frac{\lam}{k(n)}\ev\left[T_{l_{k(n)}}\right].
\]
Note that in order to apply Kolmogorov's LLN we used the fact that $l_i-l_{i-1}\le\left\lceil\frac{1}{\lam}\right\rceil$. Using again the construction of the environment $\nu$ we get that
the last expression is equal or bigger than
\[
\liminf_{n\rightarrow\infty}\frac{\lam}{k(n)}\ev\left[T^\nu_{l_{k(n)}}\right].
\]

Since in the environment $\nu$, $\lim_{n\rightarrow\infty}\frac{1}{n}\sum_{i=0}^{n-1}\ind_{\nu(i)=1}$
exists and equals $\lam$ we have that
$\lim_{n\rightarrow\infty}\frac{l_{k(n)}}{k(n)}=\frac{1}{\lam}$ and
so we get that
\begin{equation}
\begin{aligned}
\liminf_{n\rightarrow\infty}\frac{T_n}{n}
&\geq\liminf_{n\rightarrow\infty}\frac{\lam}{k(n)}\ev\left[T^\nu_{l_{k(n)}}\right]\\
&=\liminf_{n\rightarrow\infty}\frac{1}{l_{k(n)}}\ev\left[T^\nu_{l_{k(n)}}\right]\\
&\geq \liminf_{n\rightarrow\infty}\frac{\ev\left[T^\nu_n\right]}{n}.
\end{aligned}
\end{equation}

By Proposition \ref{prop:limit_proposition_Upsilon
environment} the limit
$\lim_{n\rightarrow\infty}\frac{T^{\nu}_n}{n}=\frac{1}{\lim_{n\rightarrow\infty}\frac{X^\nu_n}{n}}$
exists, which together with Fatou lemma's gives
\[\liminf_{n\rightarrow\infty}\frac{\ev\left[T^\nu_n\right]}{n}\geq
\ev\left[\liminf_{n\rightarrow\infty}\frac{T^\nu_n}{n}\right]
=
\ev\left[\lim_{n\rightarrow\infty}\frac{T^\nu_n}{n}\right]
.
\]
Finally using Lemma \ref{lem_as_constant},
$\lim_{n\rightarrow\infty}\frac{T^\nu_n}{n}$ is a $\prob_\nu$ almost
sure constant. Thus
\[
\ev\left[\lim_{n\rightarrow\infty}\frac{T^\nu_n}{n}\right]
=\lim_{n\rightarrow\infty}\frac{T^\nu_n}{n}
\]
and
\[
\liminf_{n\rightarrow\infty}\frac{T_n}{n}\geq
\lim_{n\rightarrow\infty}\frac{T^\nu_n}{n},\quad\prob_{\om,\nu}\text{ a.s}.
\]

\end{proof}


\begin{proof}[Proof of Proposition \ref{prop_Lack_ot_tightness}]
By Proposition \ref{prop:environment_to_environment} there exists a
$\Upsilon$ environment $\nu$ such that
\[
\liminf_{N\rightarrow\infty}\frac{T_N}{N}\geq \lim_{N\rightarrow\infty}\frac{T^\nu_N}{N},
\]
so it is enough to show that for every $\Upsilon$ environment $\nu$ we have
\[
\limsup_{N\rightarrow\infty}\frac{X_N}{N}\leq \lam-D\lam^3, \text{ a.s,}
\]
for some constant $D>0$. But this indeed holds since
\[
\lam-\limsup_{N\rightarrow\infty}\frac{X_N}{N}=\lam-\frac{1}{2m+1-m(m+1)\lam}=\frac{n}{mn+l}-\frac{1}{2m+1-m(m+1)\frac{n}{nm+l}},
\]
rearranging the last expression we get
\[
=\lam^3\cdot \frac{l(n-l)}{n^2}\cdot \frac{1}{1-\frac{l(n-l)}{(mn+l)^2}}.
\]
Using the fact that $l>0$ and $n>1$ we get that the last expression is bigger than
\[
\lambda^3\cdot \left(\frac{1}{n}-\frac{1}{n^2}\right)\cdot \frac{1}{1-\lambda^2\left(\frac{1}{n}-\frac{1}{n^2}\right)}\geq \lambda^3\cdot \left(\frac{1}{n}-\frac{1}{n^2}\right).
\]
\end{proof}


\section{Transience Recurrence and the triviality of the $\limsup$}

\begin{defn}
For a $(p,\lam)$ environment $\om$ we define $S(\om)$ by
\[
S(\om)=\sum_{n=1}^{\infty}\prod_{j=1}^{n}\rho(j),
\]
where as before for $j\in\BN$
\[
\rho(j)=\frac{1-\om(j)}{\om(j)}.
\]
\end{defn}
\begin{defn}
For an environment $\om$ and $x\in\BZ$ we define $\theta^x\om$ to be the translation of $\om$ by $x$ i.e. for every $n\in\BZ$, $\theta^x\om(n)=\om(n+x)$.
\end{defn}

\begin{lem} \label{transient_recurrent}
Fix a $(p,\lam)$ environment $\om$. If $S(\om)<\infty$ then a random
walk in $\om$ is  transient to the right, i.e, for
every $x_0\in\BZ$ we have
$P_\om^{x_0}\left(\lim_{n\rightarrow\infty}X_n=\infty\right)=1$. If
$S(\om)=\infty$ then a random walk in $\om$ is recurrent, i.e, for every $x_0\in\BZ$ we have
$P_\om^{x_0}\left(-\infty=\liminf_{n\rightarrow\infty}X_n<\limsup_{n\rightarrow\infty}X_n=\infty\right)=1$.
\end{lem}

\begin{proof}
This is a straight implication of the ideas and results of Theorem
2.1.2 of \cite{zeitouni2001lecture}. Note that since $\om(x)\ge\frac{1}{2}$ for all $x \in\Z$, the walk can not be transient to the left.
\end{proof}

\begin{cor}
If for a $(p,\lam)$
environment $\om$ the limit of the density exists and positive, i.e.
\[
\lim_{n\rightarrow\infty}\frac{1}{n}\sum_{i=0}^{n-1}\ind_{\{\om(i)=p\}}=\lam>0,
\]
then the random walk is transient to the right. Indeed, in this case one can fix $0<\ep<\lam$ and $x_0\in\BZ$ and then find $N\in\BN$
such that for every $n\geq N$ we have $\frac{1}{n}\sum_{i=0}^{n-1}\ind_{\{\om(i)=p\}}>\lam-\ep$. Thus for every $n\geq N$
\[
\sum_{k=1}^{n}\prod_{j=1}^{k}\rho(x+j)\leq \sum_{k=1}^{N-1}1+\sum_{k=N}^{n}{\left(\frac{1-p}{p}\right)^{(\lam-\ep)k}<N+\sum_{k=N}^\infty\left(\frac{1-p}{p}\right)^{(\lam-\ep)k}<\infty}
\]
since $\frac{1-p}{p}<1$. Therefore by taking the limit $n\rightarrow\infty$ one gets
\[
S(\theta^{x_0}\om)<\infty
\]
and so the random walk is transient to the right.
\end{cor}

Next we prove Lemma \ref{lem_as_constant}.

\begin{proof}[Proof of Lemma \ref{lem_as_constant}]
For $v\in\BR$ and $\delta>0$ we denote by $A_{v,\delta}$ the event
\[
A_{v,\delta}=\left\{\left|\limsup_{n\rightarrow\infty}\frac{X_n}{n}-v\right|<\delta\right\}.
\]

Assume that $\prob_\om^0(A_{v,\delta})>0$. Since $A_{v,\delta}\in
\sigma(X_1,X_2,\ldots)$, for every $\ep>0$, one can find $M\in\BN$
and an event $B_{v,\delta}^M\in\sigma(X_1,X_2,\ldots,X_M)$ such
that
\[\prob^0_\om(A_{v,\delta}\triangle B_{v,\delta}^M)<\epsilon.\]
Notice that for small enough $\ep>0$ this implies that
$\prob_\om^0(B_{v,\delta}^M)>\frac{\prob_\om^0(A_{v,\delta})}{2}\equiv c>0$

Since $X_n$ is a nearest neighbor random walk on $\BZ$ which starts
at the origin we have the estimate $|X_n|\leq n$, and therefore
\bae
\prob_\om^0(A_{v,\delta}\cap B_{v,\delta}^M)
&=\sum_{j=-M}^{M}\prob_\om^0(A_{v,\delta} \cap B_{v,\delta}^M \cap\{
X_M=j\})\\
&=\sum_{j=-M}^{M}\prob_\om^0(A_{v,\delta}|B_{v,\delta}^M \cap
\{X_M=j\})\cdot \prob_\om^0(B_{v,\delta}^M \cap\{ X_M=j\}). \eae
By the
Markov property of the random walk this equals to
\[
\sum_{j=-M}^{M}\prob_\om^j(A_{v,\delta})\cdot
\prob_\om^0(B_{v,\delta}^M \cap \{X_M=j\}).
\]
Dividing the last formula by $\prob_\om^0(B_{v,\delta}^M)$ we see that
\begin{equation}
\prob_\om^0(A_{v,\delta}|B_{v,\delta}^M)=\sum_{j=-M}^{M}\prob_\om^j(A_{v,\delta})\cdot
\prob_\om^0(X_M=j|B_{v,\delta}^M). \label{zero_one_1}
\end{equation}
By the choice of $B_{v,\delta}^M$ we get that \bae
\prob_\om^0(A_{v,\delta}|B_{v,\delta}^M)
&=\frac{\prob_\om^0(A_{v,\delta} \cap
B_{v,\delta}^M)}{\prob_\om^0(B_{v,\delta}^M)}=\frac{\prob_\om^0(B_{v,\delta}^M)-\prob_\om^0(B_{v,\delta}^M\backslash
A_{v,\delta})}{\prob_\om^0(B_{v,\delta}^M)}\\
&=1-\frac{\prob_\om^0(B_{v,\delta}^M\backslash
A_{v,\delta})}{\prob_\om^0(B_{v,\delta}^M)}\geq 1 -
\frac{\epsilon}{c}.
\eae
In addition we have that
\[
\sum_{j=-M}^{M}\prob_\om^0(X_M=j|B_{v,\delta}^M)=1.
\]
Using the last two observations and equation \eqref{zero_one_1} we get that for small enough $\ep>0$ there exists $M\in\BN$ and $-M\leq j \leq M$ such that
\[
\prob_\om^j(A_{v,\delta})>1-\frac{\ep}{c}>\frac{1}{2}.
\]

Assume now towards contradiction that there exist two different
values $v_1$ and $v_2$ in the support of
$\limsup_{n\rightarrow\infty}\frac{X_n}{n}$. Choose
$\delta_1,\delta_2>0$ small enough so that $A_{v_1,\delta_1}\cap
A_{v_2,\delta_2}=\emptyset$. Using the conclusion of what we
showed so far, one can find two integers $j_1$ and $j_2$ such that
\[\prob_\om^{j_1}(A_{v_1,\delta_1})>\frac{1}{2} ~~~~~~\text{and}~~~~~~ \prob_\om^{j_2}(A_{v_2,\delta_2})>\frac{1}{2}.\]
Without lost of generality we assume that $j_1<j_2$. But according
to Lemma \ref{transient_recurrent} a random walk in a $(p,\lambda)$
environment $\om$ is $\prob_\om^x$ almost surely transient to the right or $\prob_\om^x$ almost surely recurrent.
and therefore a random walk starting at $j_1$ will reach $j_2$ at
some finite random time $N$ almost surely. Consequently, if $X_n$
indeed starts at $j_1$, then
\[
\limsup_{n\rightarrow\infty}\frac{X_n}{n}=\limsup_{n\rightarrow\infty}\frac{X_{n+N}}{n+N}=\limsup_{n\rightarrow\infty}\frac{X_{n+N}}{n}.
\]
But the $limsup$ on the left is distributed according to a
random walk starting at $j_1$ and the one on the right is
distributed according to a random walk starting at $j_2$, which
gives the desired contradiction.
\end{proof}

\section{Some conjectures and questions}

In this article we studied random walks in $\BZ$ environment composed of two point types, $(\frac{1}{2}
,\frac{1}{2}\)) and $(p,1-p)$ for $p>\frac{1}{2}$. We ask for the following generalizations:
\begin{ques}
What can be said about random walks in environments of $\BZ$ composed of two types $(p,1-p)$ and $(q,1-q)$ for $\frac{1}{2}
<p<q<1$? More precisely we ask for a bound on the speed and give the following conjecture :
\end{ques}

\begin{conj}
An environment which maximize the speed is given up to some integer effect by equally spaced drifts.
\end{conj}

\begin{ques}
What can be said about the speed of random walks with more than one type of drifts? For example about environments composed of three types $(\frac{1}{2},\frac{1}{2})$, $(p,1-p)$ and $(q,1-q)$ for $\frac{1}{2}<p<q<1$.
\end{ques}


\section*{Acknowledgments}
$~$\\
The authors would like to thank Noam Berger, Ori Parzanchevski, Ran Tessler and Ofer Zeitouni for helpful discussions. We would also like to thank the comments of an anonymous referee.\\


\bibliography{ME}
\bibliographystyle{plain}

\end{document}